\documentclass[11pt,amsfonts]{article}
\usepackage[margin=1in]{geometry}  
\usepackage{graphicx}              

\usepackage{amsmath, amssymb, amsthm, epsfig}               
\usepackage{enumerate}
\usepackage{amsfonts}              
\usepackage{amsthm}                
\usepackage{amsthm}
\usepackage{enumerate}
\theoremstyle{plain}
\usepackage{color}
\newtheorem{corollary}{Corollary}[section]
\newtheorem{definition}[corollary]{Definition}

\newtheorem{lemma}[corollary]{Lemma}
\newtheorem{prp}[corollary]{Proposition}
\newtheorem{remark}[corollary]{Remark}
\newtheorem{thm}[corollary]{Theorem}

\newfont{\sBlackboard}{msbm10 scaled 900}

\newcommand{\mylabel}[1]{\label{#1}
	\ifx\undefined\stillediting
	\else \fbox{$#1$}\fi }
\newcommand{\BE}{\begin{equation}}

\newcommand{\EEQ}{\end{equation}}
\newcommand{\rfb}[1]{\mbox{\rm
		(\ref{#1})}\ifx\undefined\stillediting\else:\fbox{$#1$}\fi}

\newfont{\Blackboard}{msbm10 scaled 1200}

\newfont{\roma}{cmr10 scaled 1200}

\newcommand{\bb}{\begin{equation}}
\newcommand{\bbb}{\end{equation}}

\newcommand{\mm}    {{\hbox{\hskip 0.5pt}}}

\newcommand{\bluff} {{\hbox{\raise 15pt \hbox{\mm}}}}
scaled\magstep2

\makeatletter
\def\section{\@startsection {section}{1}{\z@}{-3.5ex plus -1ex minus
		-.2ex}{2.3ex plus .2ex}{\large\bf}}
\makeatother

\begin{document}
\title{Strauss and Lions type theorems for the fractional Sobolev spaces with variable exponent and applications to nonlocal Kirchhoff-Choquard problem}
\author{Sabri Bahrouni\footnote{
sabri.bahrouni@fsm.rnu.tn} and Hichem Ounaies
\footnote{ hichem.ounaies@fsm.rnu.tn}\\
Mathematics Department, Faculty of Sciences, University of Monastir,\\ 5019 Monastir, Tunisia}
\maketitle

\begin{abstract}
This paper deals with Strauss and Lions-type theorems for fractional Sobolev spaces with variable exponent $W^{s,p(.),\tilde{p}(.,.)}(\Omega)$, when $p$ and $\tilde{p}$ satisfies some conditions. As application, we study
the existence of solutions for a class of  Kirchhoff-Choquard problem in $\mathbb{R}^N$.
\end{abstract}



\section{Introduction}

Recently, U. Kaufmann {\it et al.} \cite{Kaufmann} introduced a new class of fractional Sobolev spaces with variable exponents $W^{s,p(.),\tilde{p}(.,.)}(\Omega)$  defined as
$$W^{s,p(.),\tilde{p}(.,.)}(\Omega)=\bigg{\{}u\in L^{p}(\Omega):\ \int_{\Omega}\int_{\Omega}\frac{|u(x)-u(y)|^{\tilde{p}(x,y)}}{|x-y|^{N+s\tilde{p}(x,y)}}dxdy<\infty\bigg{\}},$$
where $\Omega$  is a Lipschitz bounded domain in $\mathbb{R}^N$ $(N\geq2)$, $s\in(0,1)$, $p\in C(\overline{\Omega},(1+\infty))$ and $\tilde{p}\in C(\mathbb{R}^N\times\mathbb{R}^N,(1+\infty))$ satisfying $\tilde{p}(x,y)=\tilde{p}(y,x)$ for all $x,y\in\mathbb{R}^N$.\\
 With the  restriction $\tilde{p}(x,x)<p(x)$ for all $x\in\overline{\Omega}$, they proved the compact embedding $W^{s,p(.),\tilde{p}(.,.)}(\Omega)\hookrightarrow\hookrightarrow L^{r(x)}(\Omega)$ for any $r\in C(\overline{\Omega})$ satisfying $1<r(x)<\tilde{p}^{*}_{s}:=\frac{N\tilde{p}(x,x)}{N-s\tilde{p}(x,x)}$ for all $x\in\overline{\Omega}$. Using this compact embedding result, the authors proved the existence of solutions for a nonlocal problem involving the fractional $p(x)$-Laplacian by applying a direct method of Calculus of Variations.
In \cite{Kim}, the authors showed the continuous embedding of $W^{s,p(.),\tilde{p}(.,.)}(\mathbb{R}^N)\hookrightarrow L^{r(x)}(\mathbb{R}^N)$
for any $r\in C(\overline{\Omega})$ satisfying $1<r(x)<\tilde{p}^{*}_{s}=\frac{N\tilde{p}(x,x)}{N-s\tilde{p}(x,x)}$ for all $x\in\mathbb{R}^N$, and  as application, they obtained a-priori bounds and multiplicity of solutions for some nonlinear elliptic problems involving the fractional $p(x)$-Laplacian. In \cite{anouar1, anouar2}, the authors gave some further basic properties both on the function space $W^{s,p(.),\tilde{p}(.,.)}(\mathbb{R}^N)$ and the related nonlocal operator $(-\triangle)^s_{\tilde{p}(.,.)}$, where the fractional $p(x)$-Laplacian $(-\triangle)^s_{\tilde{p}(.,.)}$ is given by
$$(-\triangle)^s_{\tilde{p}(.,.)}u(x):=P.V.\int_{\mathbb{R}^N}\frac{|u(x)-u(y)|^{\tilde{p}(x,y)-2}(u(x)-u(y))}{|x-y|^{N+s\tilde{p}(x,y)}}dy,\ \text{for all}\ x\in\Omega,$$
$P.V.$ is a commonly used abbreviation in the principal value sense. Note that the operator $(-\triangle)^s_{\tilde{p}(.,.)}$ is the fractional version of the well known $p(x)-$Laplacian operator $\triangle_{p(x)}=div(|\nabla u(x)|^{p(x)-2}u(x))$ and in the constant exponent case, it is known as the fractional $p$-Laplacian operator $(-\triangle)^s_{p}$. There are just a few papers dealing with this operator, we refer the readers to \cite{Ali, Xiang1, Azroul1, Azroul2, KYHO}.\\

Many nonlinear problems of mathematical physics are posed on unbounded domains of $\mathbb{R}^N$. The unboundedness of the domain generally prevents the study of these types of problems by general methods of nonlinear analysis due to the lack of compactness.
 In many different cases, it has been observed that by restricting to sub-spaces formed by functions respecting some symmetries of the problem, some forms of compactness were obtained.\\
 Denote by
 $$W^{s,p(.),\tilde{p}(.,.)}_{rad}(\mathbb{R}^{N})=\{u\in W^{s,p(.),\tilde{p}(.,.)}(\mathbb{R}^{N}):\ u\ \text{is radially symmetric}\}.$$

  By $u$ is radially symmetric we mean, a function $u:\mathbb{R}^N\rightarrow\mathbb{R}$ satisfying $u(x)=u(y)$ for all $x,y\in\mathbb{R}^N$ such that $\|x\|=\|y\|$.
 Strauss was the first who observed that there exists an interplay between the regularity of the function and its radially symmetric property, (see \cite{Strauss}). Later on, Strauss's result was generalized in many directions, see \cite{Alves1, Fan1, Lions, daniel, giovanni1, napoli, Morry} for a survey of related results and elementary proofs of some of them within the framework of different spaces.
 To the author's best knowledge, there have no works dealing with Strauss and Lion's results for the fractional Sobolev spaces with variable exponents. Thus, the first purpose of this paper is to give a general version of a Lions-type lemma and clearly highlight the compactness result of Strauss.

 Denote
$$C_+(\overline{\Omega})=\bigg{\{}h\in C(\Omega):\ 1<\inf_{x\in\overline{\Omega}}h(x)\leq\sup_{x\in\overline{\Omega}}h(x)<\infty\bigg{\}},$$
and for $h\in C_+(\overline{\Omega})$, denote by $$h_-(\Omega)=\inf_{x\in\overline{\Omega}}h(x)\ \ \text{and}\ \ h_+(\Omega)=\sup_{x\in\overline{\Omega}}h(x).$$
In what follows, for $p,q\in C_+(\overline{\Omega})$,  we denote by $p\ll q$ in $\Omega$ provided $\inf_{x\in\Omega}(q(x)-p(x))>0$. We state now our main results.

\begin{thm}\label{Lions}({\bf  A Lions-type lemma for $\mathbf{W^{s,p(.),\tilde{p}(.,.)}}$})
   Let $0<s<1$, $p\in C_+(\mathbb{R}^N)$ be a uniformly continuous function and $\tilde{p}\in C_+(\mathbb{R}^N\times\mathbb{R}^N)$ is symmetric, i.e., $\tilde{p}(x,y) = \tilde{p}(y,x)$ for all $x,y\in \mathbb{R}^N$ satisfying, $$s\tilde{p}_+:=s\sup_{(x,y)\in\overline{\Omega}}\tilde{p}(x,y)<N\ \ \ \text{and}\ \ \ \tilde{p}(x,y)\leq p(x)\ \ \ \text{for all}\ \ (x,y)\in\mathbb{R}^N\times\mathbb{R}^N.$$
     Let $(u_n)$ be a bounded sequence of $W^{s,p(.),\tilde{p}(.,.)}(\mathbb{R}^N)$ such that
  \begin{equation}\label{6}
  \sup_{y\in\mathbb{R}^N}\int_{B(y,r)}|u_n|^{q(x)}dx\rightarrow0,\ n\rightarrow+\infty,
   \end{equation}
   for some $r>0$ and $q\in C_+(\mathbb{R}^N)$ satisfying $p\ll q\ll\tilde{p}^{*}_s.$
  Then $$u_n\rightarrow0\ \ \ \text{in}\ \ \ L^{\alpha(x)}(\mathbb{R}^N)\ \ \text{for any}\  p\ll\alpha\ll\tilde{p}^{*}_s.$$
\end{thm}

An important consequence of the previous Theorem, is the Strauss's compact embedding.

\begin{thm}\label{Strauss}({\bf A compactness result for radial functions})
  Let $0<s<1$,
   $p\in C_+(\mathbb{R}^N)$ be a uniformly continuous, radially symmetric function and $\tilde{p}\in C_+(\mathbb{R}^N\times\mathbb{R}^N)$ such that $$s\tilde{p}_+<N\ \ \ \text{and}\ \ \ \tilde{p}(x,y)\leq p(x)\ \ \ \text{for all}\ \ (x,y)\in\mathbb{R}^N\times\mathbb{R}^N.$$
   Then, for any function $\alpha\in C_+(\mathbb{R}^N)$ satisfying

  $$p\ll\alpha\ll \widetilde{p}^{*}_s:=\frac{N\tilde{p}(x,x)}{N-s\tilde{p}(x,x)},$$ 
the embedding $$W^{s,p(.),\tilde{p}(.,.)}_{rad}(\mathbb{R}^N)\hookrightarrow L^{\alpha(x)}(\mathbb{R}^N),$$ is compact.
\end{thm}

As a second goal of this paper, we investigate the existence of solution of the following class of Kirchhoff-Choquard equation
\begin{equation}\label{eq1}
\begin{cases}
M\big{(}X(u)\big{)}
(-\triangle)^s_{\tilde{p}(.,.)}u(x)=\bigg{(}\displaystyle\int_{\mathbb{R}^N}
\displaystyle\frac{F(y,u(y))}{|x-y|^{\lambda(x,y)}}dy\bigg{)}f(x,u(x))\ \ \text{in}\ \ \mathbb{R}^N, \\
 \ \ \ \ \ \ \ u\in W^{s,p(.),\tilde{p}(.,.)}(\mathbb{R}^N),
 \end{cases}
 \end{equation}
where $$X(u):=\int_{\mathbb{R}^N}\int_{\mathbb{R}^N}\frac{|u(x)-u(y)|^{\tilde{p}(x,y)}}{\tilde{p}(x,y)|x-y|^{N+s\tilde{p}(x,y)}}dxdy,$$
$M:\mathbb{R}^+_0\rightarrow\mathbb{R}^+$ is a continuous function, $\lambda:\mathbb{R}^N\times\mathbb{R}^N\rightarrow\mathbb{R}$ is symmetric function, i.e., $\lambda(x,y)=\lambda(y,x)$ for all $(x,y)\in\mathbb{R}^N\times\mathbb{R}^N$ with $0<\lambda_{-}\leq\lambda_{+}<N$ and $f:\mathbb{R}^N\times\mathbb{R}\rightarrow\mathbb{R}$ is Carath\'eodory functions, $F(x,t)$ is the primitive of $f(x,t)$, that is, $$F(x,t)=\int_{0}^{t}f(x,\tau)d\tau.$$

We say that $\tilde{p}$ is radially symmetric in $\mathbb{R}^N\times\mathbb{R}^N$, if
\begin{equation}\label{tilderadial}
  \tilde{p}(\|x_1\|,\|y_1\|)=\tilde{p}(\|x_2\|,\|y_2\|),
\end{equation}
for all $x_1,x_2,y_1,y_2\in\mathbb{R}^N$ and $\|x_1\|=\|x_2\|$, $\|y_1\|=\|y_2\|$.\\

The nonlinearity on the right side of \eqref{eq1} is motivated by the Choquard equation $$-\triangle u+u=(I_{\alpha}\ast|u|^p)|u|^{p-2}u\ \ \text{in}\ \mathbb{R}^N.$$
Here, $I_{\alpha}$ is the Riesz potential of order $\alpha\in(0,N)$  defined for each point $x\in\mathbb{R}^N\setminus\{0\}$ by $$I_{\alpha}(x)=\frac{A_{\alpha}}{|x|^{N-\alpha}}\ \ \ \ \text{where}\ \ \ A_{\alpha}=\frac{\Gamma(\frac{N-\alpha}{2})}{\Gamma(\frac{\alpha}{2})\pi^{N/2}2^{\alpha}}.$$
We mention \cite{Lieb, Lions1, Menzela,Moroz,LIXIA WANG, Tang} for Choquard problems involving the classical Laplacian operator.\\

More recently, in \cite{Alves}, Alves and Tavares  have proved a Hardy-Littlewood-Sobolev inequality for variable
exponents. They applied this inequality together with variational method to
establish the existence of solutions for a class of the following Choquard equations involving the
p(x)-Laplacian operator
$$\begin{cases}
(-\triangle)_{p(x)}u(x)=\bigg{(}\displaystyle\int_{\mathbb{R}^N}
\displaystyle\frac{F(y,u(y))}{|x-y|^{\lambda(x,y)}}dy\bigg{)}f(x,u(x))\ \ \text{in}\ \ \mathbb{R}^N, \\
 \ \ \ \ \ \ \ u\in W^{1,p(x)}(\mathbb{R}^N).
 \end{cases}$$

For $M(t)=1$,  problem \eqref{eq1} was investigated in \cite{Biswas}. The authors used a variational methods
to show the existence of solutions. This method works well thanks to a Hardy-Littlewood-Sobolev type inequality that has the following statement.

\begin{thm}\label{Hardy}[\cite{Biswas}]
  Let $0<s<1$ and $\tilde{p}$ is uniformly continuous in $\mathbb{R}^N\times\mathbb{R}^N$.\\ Let $h\in C_{+}(\mathbb{R}^N\times\mathbb{R}^N)$ such that $$\frac{2}{h(x,y)}+\frac{\lambda(x,y)}{N}=2\ \ \text{for all}\ \ (x,y)\in\mathbb{R}^N\times\mathbb{R}^N,$$
 and $r\in C_{+}(\mathbb{R}^N)\cap\mathcal{C}_h$, where $$\mathcal{C}_h:=\{p(x)\leq r(x)\ h_{-}\leq r(x)\ h_{+}<\tilde{p}^{*}_s:=\frac{N\tilde{p}(x,x)}{N-s\tilde{p}(x,x)}\ \text{for all}\ x\in\mathbb{R}^N\}.$$ Then, for $u\in W^{s,p(.),\tilde{p}(.,.)}(\mathbb{R}^N)$, $|u|^{r(.)}\in L^{h_{+}}(\mathbb{R}^N)\cap L^{h_{-}}(\mathbb{R}^N)$ such that,
 $$\int_{\mathbb{R}^N}\int_{\mathbb{R}^N}\frac{|u(x)|^{r(x)}|u(y)|^{r(y)}}{|x-y|^{\lambda(x,y)}}dxdy\leq C\big{(}||u|^{r(.)}|_{h_{+},\mathbb{R}^N}^{2}+||u|^{r(.)}|_{h_{-},\mathbb{R}^N}^{2}\big{)}$$ where $C>0$ is a constant independent of $u$.
\end{thm}

 Throughout this paper, the Kirchhoff function $M$ related to the elliptic part of the problem \eqref{eq1}, satisfies:

\begin{enumerate}
  \item [$(M_1)$] There exists $m_0>0$ such that $M(t)\geq m_0$ for all $t>0$.
  \item [$(M_2)$] There exists $\alpha>0$ such that $M(t)t\leq\alpha\mathcal{M}(t)$ where $\mathcal{M}(t)=\displaystyle\int_{0}^{t}M(s)ds$.
\end{enumerate}
Obviously, a typical prototype for M,  is given by $M(t)=a+bt^m$ with $m,a>0$ and $b\geq0$ for all $t\geq0$. When $M$ is this form, we say that problem \eqref{eq1} is non-degenerate, while it is called degenerate if $a = 0$ and $b > 0$.\\

The functions $f$ and $F$ satisfy:
\begin{enumerate}
  \item [$(F_1)$] There exists a constant $L>0$, $h\in C_{+}(\mathbb{R}^N\times\mathbb{R}^N)$ and a function $r\in C_{+}(\mathbb{R}^N)\cap\mathcal{C}_h$ with $r_{-}>\tilde{p}_{+}$ such that
  $$|f(x,t)|\leq L\ |t|^{r(x)-1}\ \ \ \text{for all}\ \ \ (x,t)\in\mathbb{R}^N\times\mathbb{R}.$$
  \item [$(F_2)$] There exist a constant $\theta>\alpha\tilde{p}_{+}$ such that $0<\theta F(x,t)\leq 2t f(x,t)$ for all $t>0$.
\end{enumerate}

We state the following result.
\begin{thm}\label{thm1}
Assume that $(M_1)-(M_2)$ and $(F_1)-(F_2)$ hold, $p,\ \lambda$ and $F$ are radial symmetric functions and $\tilde{p}$ is uniformly continuous satisfies \eqref{tilderadial}. Then the problem \eqref{eq1} admits a nontrivial radial solution.
\end{thm}

The main tools that we employed for the search of a weak solution of problem \eqref{eq1}, are the Mountain pass theorem due to Ambrosetti and Rabibowitz \cite{Rabino}, Strauss-Lions type for the space $W^{s,p(.),\tilde{p}(.,.)}_{rad}(\mathbb{R}^N)$ proved in section $3$ and the  principle of symmetric criticality, given in \cite{Palais}. Theorem \ref{thm1} extends the result made in \cite{Alves}, in the sense that we are dealing with fractional version of the Sobolev space with variable exponents.\\

This paper is organized as follows. In Section $2$, we give some definitions and fundamental properties of the spaces
$L^{p(x)}(\Omega)$ and $W^{s,p(.),\tilde{p}(.,.)}(\Omega)$. In Section $3$,  we prove Theorem \ref{Strauss} and Theorem \ref{Lions}. Finally, in Section $4$, we introduce our abstract framework related to problem \eqref{eq1} and we prove Theorem \ref{thm1}.

\section{Preliminaries}

In this section, we briefly review the definitions and list some basic properties of the
Lebesgue spaces with variable exponent and the fractional Sobolev spaces with variable exponent.

Let $\Omega$ be a Lipschitz domain in $\mathbb{R}^N$ and $p\in C_+(\overline{\Omega})$. We define the variable exponent Lebesgue space
$$L^{p(x)}(\Omega):=\bigg{\{}u:\Omega\rightarrow\mathbb{R},\ \text{measurable}\ \text{and}\ \int_{\Omega}|u(x)|^{p(x)}dx<\infty\bigg{\}},$$
endowed with the Luxemburg norm $$|u|_{p,\Omega}:=\inf\bigg{\{}\lambda>0:\  \int_{\Omega}\bigg{|}\frac{u(x)}{\lambda}\bigg{|}^{p(x)}dx\leq1 \bigg{\}}.$$
Some basic properties of $L^{p(x)}(\Omega)$ are listed in the following three propositions.

\begin{prp}\label{prp1}[\cite{Diening}, Corollary 3.3.4]
Let $\alpha,\beta\in C_+(\overline{\Omega})$ such that $\alpha(x)\leq\beta(x)$ for all $x\in\overline{\Omega}$. Then, we have
$$|u|_{\alpha,\Omega}\leq2(1+|\Omega|)|u|_{\beta,\Omega}\ \text{for all}\ u\in  L^{\beta(x)}(\Omega).$$
\end{prp}
\begin{prp}\label{prp2}[\cite{Fan, Kova}]
  The space $L^{p(x)}(\Omega)$ is a separable, uniformly convex Banach space and its dual space is $L^{p^{'}(x)}(\Omega)$, where $\frac{1}{p(x)}+\frac{1}{p^{'}(x)}=1$. Furthermore, for any $u\in L^{p(x)}(\Omega)$ and $v\in L^{p^{'}(x)}(\Omega)$, we have $$\bigg{|}\int_{\Omega}uv\ dx\bigg{|}\leq 2|u|_{p,\Omega}|v|_{p^{'},\Omega}.$$
\end{prp}
\begin{prp}\label{prp3}[\cite{Fan}]
  Define the modular $\rho:L^{p(x)}(\Omega)\rightarrow\mathbb{R}$ by $$\rho(u):=\int_{\Omega}|u|^{p(x)}dx,\ \text{for all}\ u\in L^{p(x)}(\Omega).$$
  Then, we have the following relations between norm and modular.\\
  \begin{enumerate}
    \item [(i)] $u\in L^{p(x)}(\Omega)\setminus\{0\}$ if and only if $\rho\big{(}\frac{u}{|u|_{p,\Omega}}\big{)}=1$.
    \item [(ii)] $\rho(u)>1\ (=1,\ <1)$  if and only if $|u|_{p,\Omega}>1,\ (=1,\ <1)$, respectively.
    \item [(iii)] If $|u|_{p,\Omega}>1$, then $|u|_{p,\Omega}^{p_-}\leq\rho(u)\leq |u|_{p,\Omega}^{p_+}$.
    \item [(iv)]  If $|u|_{p,\Omega}<1$, then $|u|_{p,\Omega}^{p_+}\leq\rho(u)\leq |u|_{p,\Omega}^{p_-}$.
    \item [(v)] For a sequence $(u_n)\subset L^{p(x)}(\Omega)$ and $u\in L^{p(x)}(\Omega)$, we have $$\lim_{n\rightarrow+\infty}|u_n-u|_{p,\Omega}=0\ \ \text{if and only if}\ \ \lim_{n\rightarrow+\infty}\rho(u_n-u)=0.$$
  \end{enumerate}
\end{prp}

Next, we recall  the definition and some embedding results on fractional Sobolev spaces with variable exponent that was first introduced in \cite{Kaufmann}.

Let $s\in(0,1)$ and $\tilde{p}\in C(\overline{\Omega}\times\overline{\Omega})$  be such that $\tilde{p}$ is symmetric, i.e., $\tilde{p}(x,y) = \tilde{p}(y,x)$ for all $x,y\in \overline{\Omega}$ and
 $$ 1<\tilde{p}_-(\Omega):=\inf_{(x,y)\in\overline{\Omega}\times\overline{\Omega}}\tilde{p}(x,y)\leq \tilde{p}_+(\Omega):=\sup_{(x,y)\in\overline{\Omega}\times\overline{\Omega}}\tilde{p}(x,y)<+\infty.$$

For $p\in C_+(\overline{\Omega})$, we define
$$W^{s,p(.),\tilde{p}(.,.)}(\Omega)=\bigg{\{}u\in L^{p(x)}(\Omega):\ \int_{\Omega}\int_{\Omega}\frac{|u(x)-u(y)|^{\tilde{p}(x,y)}}{|x-y|^{N+s\tilde{p}(x,y)}}dxdy<\infty\bigg{\}},$$
and for $u\in W^{s,p(.),\tilde{p}(.,.)}(\Omega)$, set $$[u]_{s,\tilde{p}(.,.),\Omega}=\inf\bigg{\{}\lambda>0:\ \rho_{\tilde{p}}\bigg{(}\frac{u}{\lambda}\bigg{)} <1\bigg{\}},$$
where $$\rho_{\tilde{p}}(u)=\int_{\Omega}\int_{\Omega}\frac{|u(x)-u(y)|^{\tilde{p}(x,y)}}{|x-y|^{N+s\tilde{p}(x,y)}}dxdy.$$
Then, $W^{s,p(.),\tilde{p}(.,.)}(\Omega)$  endowed with the norm $$\|u\|_{s,p,\tilde{p},\Omega}:=|u|_{p,\Omega}+[u]_{s,\tilde{p}(.,.),\Omega},$$
is a separable and reflexive Banach space (see \cite{anouar1, anouar2, Kaufmann}).
 We may define another norm of $W^{s,p(.),\tilde{p}(.,.)}(\Omega)$ that is $$|u|_{s,p,\tilde{p},\Omega}:=\inf\bigg{\{}\lambda>0:\ \tilde{\rho}\bigg{(}\frac{u}{\lambda}\bigg{)}<1\bigg{\}},$$ where $$\tilde{\rho}(u)=\int_{\Omega}|u|^{p(x)}dx+\int_{\Omega}\int_{\Omega}\frac{|u(x)-u(y)|^{\tilde{p}(x,y)}}{|x-y|^{N+s\tilde{p}(x,y)}}dxdy.$$ It is easy to see that $|.|_{s,p,\tilde{p},\Omega}$ is an equivalent norm of  $\|.\|_{s,p,\tilde{p},\Omega}$ and more precisely
\begin{equation}\label{15}
 \frac{1}{2} \|.\|_{s,p,\tilde{p},\Omega}\leq|.|_{s,p,\tilde{p},\Omega}\leq 2\|.\|_{s,p,\tilde{p},\Omega}.
\end{equation}
We state below some relations between the norm $|.|_{s,p,\tilde{p},\Omega}$ and the modular $\tilde{\rho}$.

\begin{prp}\label{prp5}(\cite{Kim})
  On $W^{s,p(.),\tilde{p}(.,.)}(\Omega)$  it holds that
  \begin{enumerate}
    \item [(i)] $u\in W^{s,p(.),\tilde{p}(.,.)}(\Omega)\setminus\{0\}$ if and only if $\tilde{\rho}\big{(}\frac{u}{|u|_{s,p,\tilde{p},\Omega}}\big{)}=1$.
    \item  [(ii)] $\tilde{\rho}(u)>1$ $(=1,\ <1)$  if and only if $|u|_{s,p,\tilde{p},\Omega}>1$  $(=1,\ <1)$,  respectively.
  \end{enumerate}
  Moreover,

  \begin{enumerate}
    \item  [(iii)] If $|u|_{s,p,\tilde{p},\Omega}\geq1$, then $|u|_{s,p,\tilde{p},\Omega}^{\tilde{p}_-(\Omega)}\leq \tilde{\rho}(u)\leq |u|_{s,p,\tilde{p},\Omega}^{\tilde{p}_+(\Omega)}$.
    \item  [(iv)] If $|u|_{s,p,\tilde{p},\Omega}<1$, then $|u|_{s,p,\tilde{p},\Omega}^{\tilde{p}_+(\Omega)}\leq \tilde{\rho}(u)\leq |u|_{s,p,\tilde{p},\Omega}^{\tilde{p}_-(\Omega)}$.
  \end{enumerate}
\end{prp}

\begin{thm}\label{embedding in frac}(\cite{Kim})
 Assume that $$s\tilde{p}_+(\Omega)<N\ \ \text{and}\ \ p(x)\geq\tilde{p}(x,x),\ \text{for all}\ x\in\overline{\Omega}.$$
Then, it holds that
\begin{enumerate}
  \item the embedding $W^{s,p(.),\tilde{p}(.,.)}(\mathbb{R}^N)\hookrightarrow L^{r(x)}(\mathbb{R}^N)$ is continuous for any $r\in C_+(\mathbb{R}^N)$ satisfying $p(x)\leq r(x)\ll\tilde{p}^{*}_s:=\frac{N\tilde{p}(x,x)}{N-s\tilde{p}(x,x)}$ for all $x\in \mathbb{R}^N$.
  \item if $\Omega$ is bounded, then the embedding $W^{s,p(.),\tilde{p}(.,.)}(\Omega)\hookrightarrow\hookrightarrow L^{r(x)}(\Omega)$ is compact  for any $r\in C_+(\Omega)$ satisfying $r(x)<\tilde{p}^{*}_s$ for all $x\in \overline{\Omega}$.
\end{enumerate}
\end{thm}

\section{Proof of Theorem \ref{Lions} and Theorem \ref{Strauss}}

We start by recalling the two following lemmas that will be useful in the sequels.

\begin{lemma}\label{lem1}(\cite{Fan1})
  Let $\Omega$ be an open domain in $\mathbb{R}^N$, $q,\alpha,\beta\in C_+(\mathbb{R}^N)$ such that
   \begin{equation}\label{2}
    q\ll\alpha\ll\beta,\ \ \text{in}\ \Omega.
  \end{equation}
  If $u\in  L^{\beta(x)}(\Omega)$, then $u\in L^{\alpha(x)}(\Omega)$ and
  \begin{equation}\label{5}
  \int_{\Omega}|u|^{\alpha(x)}dx\leq 2\big{|}|u|^{\alpha_1(x)}\big{|}_{m(x)}.\big{|}|u|^{\alpha_2(x)}\big{|}_{m^{'}(x)},
  \end{equation}
   where $$\alpha_1(x)=\frac{q(x)(\beta(x)-\alpha(x))}{\beta(x)-q(x)},\ \ \ \ \alpha_2(x)=\frac{\beta(x)(\alpha(x)-q(x))}{\beta(x)-q(x)}$$
   $$m(x)=\frac{\beta(x)-q(x)}{\beta(x)-\alpha(x)},\ \ \ \ m^{'}(x)=\frac{\beta(x)-q(x)}{\alpha(x)-q(x)}.$$
\end{lemma}
\begin{lemma}\label{lem2}(\cite{Fan1})
  Let $p,\tilde{p},q,r,u_n$ be given in Theorem \ref{Lions}. If $\alpha\in C_+(\mathbb{R}^N)$ satisfies
  \begin{equation}\label{1}
    q\ll\alpha\ll\tilde{p}^{*}_s,\ \ \text{in}\ \mathbb{R}^N.
  \end{equation}
  Then, $$\sup_{y\in\mathbb{R}^N}\int_{B(y,r)}|u_n|^{\alpha(x)}dx\rightarrow0,\ \ n\rightarrow+\infty.$$
\end{lemma}
Now we proceed to the proof of theorem \ref{Lions}.
\begin{proof}[{\bf Proof of Theorem \ref{Lions}}]
 According to Proposition \ref{prp3}, it is sufficient to show that $\displaystyle\int_{\mathbb{R}^N}|u_n(x)|^{\alpha(x)}dx\rightarrow0$  as  $n\rightarrow+\infty$. We deal with it in two cases.\\

   Case 1: Suppose that $\alpha$ satisfies \eqref{1}, i.e,  $$q\ll\alpha\ll\tilde{p}^{*}_s,\ \ \text{in}\ \mathbb{R}^N.$$ Let \begin{equation}\label{3}
                                                                                                                  \beta(x)=\alpha(x)+\epsilon,
                                                                                                                \end{equation}
   where $\epsilon$ is small enough such that $q\ll \beta\ll\tilde{p}^{*}_s$ in $\mathbb{R}^N$.  Then by Theorem \ref{embedding in frac}, for any $y\in\mathbb{R}^N$ and $r>0$, we have the following continuous
embedding $$W^{s,p(.),\tilde{p}(.,.)}(B(y,r))\hookrightarrow L^{\beta(x)}(B(y,r)).$$
Therefore, there exists a constant $c>0$ such that
\begin{equation}\label{4}
  |u|_{\beta(x),B(y,r)}\leq c |u|_{s,p,\tilde{p},B(y,r)},\ \ \text{for all}\ \ u\in W^{s,p(.),\tilde{p}(.,.)}(B(y,r)).
\end{equation}
Lemma \ref{lem1} gives inequality \eqref{5}. Set $$\lambda_n=\lambda_n(y,r)=\big{|}|u_n|^{\alpha_1(x)}\big{|}_{m(x),B(y,r)},\ (\alpha_1,\ m\ \text{are given by Lemma \ref{lem1}}).$$
  Combining conditions $(ii)$ of Proposition \ref{prp3}, \eqref{6} and since $\alpha_1m=q$, we obtain, $\lambda_n<1$, for $n$ sufficiently large. Applying $(iv)$ of Proposition \ref{prp3}, we get for such $n$,
 \begin{equation}\label{7}
   \lambda_n^{m_+(B(y,r))}\leq \int_{B(y,r)}|u_n|^{q(x)}dx.
 \end{equation}

 Set $$\tau_n=\sup_{y\in \mathbb{R}^N}\int_{B(y,r)}|u_n|^{q(x)}dx.$$
 For $n$ sufficiently large, we have
       \begin{equation}\label{8}
        \lambda_n\leq \tau_n^{\frac{1}{m_+(B(y,r))}}\leq\tau_n^{\frac{1}{m_+(\mathbb{R}^N)}}.
      \end{equation}
     By Lemma \ref{lem2}, it yields that $$\sup_{y\in\mathbb{R}^N}\int_{B(y,r)}|u_n|^{\beta(x)}dx\rightarrow0,\ \ n\rightarrow+\infty.$$
      Set $$\mu_n=\mu_n(y,r)=\big{|}|u_n|^{\alpha_2(x)}\big{|}_{m^{'}(x),B(y,r)},\ \ \ \ \ \sigma_n=\sigma_n(y,r)=|u_n|_{\beta(x),B(y,r)}.$$
      Hence, for $n$ sufficiently large, $\mu_n<1$ and $\sigma_n<1$.\\

        Furthermore,
      \begin{equation}\label{9}
        \mu_n^{m^{'}_+(B(y,r))}\leq \int_{B(y,r)}|u_n|^{\beta(x)}dx\leq \sigma_n^{\beta_{-}(B(y,r))}\leq \sigma_n^{\alpha_{-}(B(y,r))}.
      \end{equation}

      Inequalities \eqref{4} and \eqref{9} give that
      \begin{align}\label{11}
        \mu_n& \leq \bigg{(}\int_{B(y,r)}|u_n|^{\beta(x)}dx\bigg{)}^{\frac{1}{m^{'}_+(B(y,r))}}\leq c^{\frac{\alpha_-(B(y,r))}{m^{'}_+(B(y,r))}} |u_n|_{s,p,\tilde{p},B(y,r)}^{\frac{\alpha_-(B(y,r))}{m^{'}_+(B(y,r))}}\nonumber\\
             &\leq c_0 |u_n|_{s,p,\tilde{p},B(y,r)}^{\frac{\alpha_-(B(y,r))}{m^{'}_+(B(y,r))}}.
      \end{align}

      {\bf Claim:}
      We claim that
      $$\frac{\alpha_-(B(y,r))}{\tilde{p}_+(B(y,r))m^{'}_+(B(y,r))}\geq1\ \ \text{for}\ r\ \text{and}\ \epsilon\ \text{in}\ \eqref{3}\ \text{be small enough}.$$

      In fact,  since $p$ is uniformly continuous and $\inf_{x\in\mathbb{R}^N}(\alpha(x)-p(x))=d>0$,
       there exists $\beta>0$, such that for all $x,z\in B(y,r)$, $|x-z|<\beta$, we have $$|p(z)-p(x)|\leq\frac{d}{2}.$$
       Which yields that $$p(z)+\frac{d}{2}\leq p(x)+d\leq\alpha(x),\ \ \text{for all}\ x,z\in B(y,r).$$
      It holds that, for $r$ small enough,  $$\alpha_-(B)\geq p_+(B) +\frac{d}{2}\geq p_-(B) +\frac{d}{2}.$$  Since $\tilde{p}(x,y)\leq p(x)$, then $\tilde{p}_+(B(y,r))\leq p_+(B(y,r))$, and so $$\frac{\alpha_-(B(y,r))}{\tilde{p}_+(B(y,r))}\geq1+\frac{d}{2\tilde{p}_+(B(y,r))}.$$

       On the other hand, $q(x)\ll\alpha(x)$ in $\mathbb{R}^N$, i.e. $\inf_{x\in\mathbb{R}^N}(\alpha(x)-q(x))=d^{'}>0$. \\
       Then, for $\epsilon$ small enough,
      $$m^{'}(x)=1+\frac{\epsilon}{\alpha(x)-q(x)}\leq 1+\frac{\epsilon}{d^{'}}.$$ Hence, for such $r$ and $\epsilon$, $\frac{\alpha_-(B(y,r))}{\tilde{p}_+(B(y,r))m^{'}_+(B(y,r))}\geq1.$ Thus the claim.\\

     By Proposition \ref{prp5}, for all $u\in W^{s,p(.),\tilde{p}(.,.)}(\mathbb{R}^N)$, we have

     \begin{equation}\label{sa1}
       |u|_{s,p,\tilde{p},B(y,r)}^{\frac{\alpha_-(B(y,r))}{m^{'}_+(B(y,r))}}\leq \tilde{\rho}(u)^{\frac{\alpha_-(B(y,r))}{m^{'}_+(B(y,r))\widetilde{p}_+(B(y,r))}}\leq\tilde{\rho}(u)\ \ \text{if}\ |u|_{s,p,\tilde{p},B(y,r)}<1,
     \end{equation}
     and
     \begin{equation}\label{sa2}
       |u|_{s,p,\tilde{p},B(y,r)}^{\frac{\alpha_-(B(y,r))}{m^{'}_+(B(y,r))}}\leq \tilde{\rho}(u)^{\frac{\alpha_-(B(y,r))}{m^{'}_+(B(y,r))\widetilde{p}_-(B(y,r))}}\leq\tilde{\rho}(u)^{\alpha_+(\mathbb{R}^N)}\ \ \text{if}\ |u|_{s,p,\tilde{p},B(y,r)}>1.
     \end{equation}
Combining \eqref{sa1} and \eqref{sa2}, we get
     \begin{align}\label{12}
       |u_n|_{s,p,\tilde{p},B(y,r)}^{\frac{\alpha_-(B(y,r))}{m^{'}_+(B(y,r))}}& \leq \tilde{\rho}(u_n)+\tilde{\rho}(u_n)^{\alpha_+(\mathbb{R}^N)}.
     \end{align}
     By \eqref{5}, \eqref{8}, \eqref{11} and \eqref{12}, it follows immediately that
   \begin{align}\label{13}
     \int_{B(y,r)}|u_n|^{\alpha(x)}dx &\leq 2\tau_n^{\frac{1}{m_+(\mathbb{R}^N)}}\mu_n \nonumber\\
     &\leq2c_0
     \tau_n^{\frac{1}{m_+(\mathbb{R}^N)}}\bigg{(}\tilde{\rho}(u_n)^{\alpha_+(\mathbb{R}^N)}+\tilde{\rho}(u_n)\bigg{)}.
   \end{align}
   Covering $\mathbb{R}^N$ by a sequence of balls $B(y_i,r)$, $\{i=1,2,3,....\}$ in such a way that each point of $\mathbb{R}^N$ is contained in at most $2^N$ balls, from \eqref{13} we obtain
\begin{align}\label{14}
  \int_{\mathbb{R}^N}|u_n|^{\alpha(x)}dx& \leq \sum_{i=1}^{\infty}\int_{B(y_i,r)}|u_n|^{\alpha(x)}dx\nonumber\\
                                  &\leq 2c_0\tau_n^{\frac{1}{m_+(\mathbb{R}^N)}}
                                  \sum_{i=1}^{\infty}\bigg{[}\bigg{(}\int_{B(y_i,r)}|u_n|^{p(x)}dx+
                                  \int_{B(y_i,r)}\int_{B(y_i,r)}
                                  \frac{|u(x)-u(z)|^{\tilde{p}(x,z)}}{|x-z|^{N+s\tilde{p}(x,z)}}dxdz\bigg{)}\nonumber\\
                                  &+\bigg{(}\int_{B(y_i,r)}|u_n|^{p(x)}dx+\int_{B(y_i,r)}\int_{B(y_i,r)}
                                  \frac{|u(x)-u(z)|^{\tilde{p}(x,z)}}{|x-z|^{N+s\tilde{p}(x,z)}}dxdz\bigg{)}^{\alpha_+(\mathbb{R}^N)}\bigg{]}\nonumber\\
                                  &\leq 2c_0\tau_n^{\frac{1}{m_+(\mathbb{R}^N)}}\mathbf{\bigg{[}}\int_{\mathbb{R}^N}|u_n|^{p(x)}\sum_{i=1}^{\infty}\chi_{B(y_i,r)}(x)dx
                                  \nonumber\\
                                  &+
                                  \int_{\mathbb{R}^N}\int_{\mathbb{R}^N}
                                  \frac{|u(x)-u(z)|^{\tilde{p}(x,z)}}{|x-z|^{N+s\tilde{p}(x,z)}}\sum_{i=1}^{\infty}\chi_{B(y_i,r)\times B(y_i,r)}(x,z)dxdz\nonumber\\
                                  &+\sum_{i=1}^{\infty}\bigg{(}\int_{\mathbb{R}^N}|u_n|^{p(x)}\chi_{B(y_i,r)}(x)dx+\int_{\mathbb{R}^N}\int_{\mathbb{R}^N}
                                  \frac{|u(x)-u(z)|^{\tilde{p}(x,z)}}{|x-z|^{N+s\tilde{p}(x,z)}}\chi_{B(y_i,r)\times B(y_i,r)}(x,z)dxdz\bigg{)}^{\alpha_+(\mathbb{R}^N)}\mathbf{\bigg{]}}\nonumber\\
                                  &\leq2c_0\tau_n^{\frac{1}{m_+(\mathbb{R}^N)}}\bigg{(}(2^N)\widetilde{\rho}(u_n)
                                  +(2^N)^{\alpha_+(\mathbb{R}^N)}\widetilde{\rho}(u_n)^{\alpha_+(\mathbb{R}^N)}\bigg{)}.
\end{align}
By virtue of Proposition \ref{prp5} and the boundedness of $(u_n)_n\subset W^{s,p(.),\tilde{p}(.,.)}(\mathbb{R}^N)$, we can assert that $(\widetilde{\rho}(u_n))_n$  is bounded. Then \eqref{6} and \eqref{14} show that $$\int_{\mathbb{R}^N}|u_n|^{\alpha(x)}dx\rightarrow0,\ \ n\rightarrow+\infty.$$ Thus $u_n \rightarrow0$ in $L^{\alpha(x)}(\mathbb{R}^N)$.\\

Case 2: Let $p\ll\alpha\ll \tilde{p}^{*}_s$ without $q\ll\alpha$.\\

Let $\beta\in C_+(\mathbb{R}^{N})$ be measurable function such that $\alpha\ll\beta$ and $q\ll\beta\ll\tilde{p}^{*}_s$.  Applying Case 1 with $\beta$, we get $\displaystyle\int_{\mathbb{R}^N}|u_n|^{\beta(x)}dx\rightarrow0$. Notice that $p\ll\alpha\ll\beta$,  by Lemma \ref{lem1}, it comes that
\begin{align}\label{9o}
  \int_{\mathbb{R}^N}|u_n|^{\alpha(x)}dx & \leq 2\big{|}|u_n|^{\alpha_1(x)}\big{|}_{m(x)}.\big{|}|u_n|^{\alpha_2(x)}\big{|}_{m^{'}(x)}\nonumber\\
                                         &\leq \bigg{(}\int_{\mathbb{R}^N}|u_n|^{p(x)}dx\bigg{)}^{m_-} \bigg{(}\int_{\mathbb{R}^N}|u_n|^{\beta(x)}dx\bigg{)}^{m^{'}_-}.
\end{align}

 Since $\displaystyle\int_{\mathbb{R}^N}|u_n|^{p(x)}dx$ is bounded, by \eqref{9o} we deduce that $$\int_{\mathbb{R}^N}|u_n|^{\alpha(x)}dx\rightarrow0,\ \ n\rightarrow+\infty.$$  Thus the proof of Theorem \ref{Lions} is completed.
\end{proof}

Let's turn to the proof of Theorem \ref{Strauss}.

\begin{proof}[{\bf Proof of Theorem \ref{Strauss}}]
Let $(u_n)_{n\in\mathbb{N}}\subset W^{s,p(.),\tilde{p}(.,.)}_{rad}(\mathbb{R}^N)$ be a bounded sequence. Since $W^{s,p(.),\tilde{p}(.,.)}_{rad}(\mathbb{R}^N)$ is a reflexive space, up to subsequence, still denoted by $u_n$, we have  $$u_n\rightharpoonup0\ \ \ \text{in}\ \ \  W^{s,p(.),\tilde{p}(.,.)}_{rad}(\mathbb{R}^N).$$ Our goal is to prove that, for any $p\ll\alpha\ll \widetilde{p}^{*}_s$, $$|u_n|_{\alpha,\mathbb{R}^N}\rightarrow0\ \ \ \text{as}\ \ \ n\rightarrow+\infty.$$
By Proposition \ref{prp3}, we get that $\big{(}\rho(u_n)\big{)}$ is bounded, that is, for some $c>0$,
               \begin{equation}\label{17}
                 \int_{\mathbb{R}^N}|u_n(x)|^{p(x)}dx<c.
               \end{equation}
Let $r>0$, as, for all $n\in\mathbb{N}$, $u_n$ and $p$ are all radially symmetric, then
\begin{equation}\label{18}
  \int_{B(y_1,r)}|u_n(x)|^{p(x)}dx=\int_{B(y_2,r)}|u_n(x)|^{p(x)}dx\ \ \ \text{for all}\ \ \ y_1,y_2\in\mathbb{R}^N,\ \ |y_1|=|y_2|.
\end{equation}
  In the sequel, we proceed as in the proof of Theorem 3.1 in \cite{Fan}. For each $y\in\mathbb{R}^N$, $|y|>r$ denote by $\gamma(y)$ the maximum of those integers $j\geq1$ such that there exist $y_1,y_2,\ldots,y_j\in\mathbb{R}^N$, with $|y_1|=|y_2|=\ldots=|y_j|=|y|$ and $$B(y_i,r)\cap B(y_k,r)=\varnothing\ \ \  \text{whenever}\ i\neq k.$$ It is easy to see that \begin{equation}\label{16}
 \gamma(y)\rightarrow+\infty\ \ \ \text{as}\ \ \ |y|\rightarrow+\infty.
 \end{equation}
 For each $y\in\mathbb{R}^N$, $|y|>r$, we choose $y_1,...,y_{\gamma(y)}\in\mathbb{R}^N$ as above,
 then by \eqref{17}, \eqref{18} and \eqref{16}, we get
 \begin{align*}
   c & \geq \int_{\mathbb{R}^N}|u_n(x)|^{p(x)}dx\geq \sum_{i=1}^{\gamma(y)}\int_{B(y_i,r)}|u_n(x)|^{p(x)}dx\\
   &\geq \gamma(y) \int_{B(y,r)}|u_n(x)|^{p(x)}dx.
 \end{align*}
 Therefore $$\int_{B(y,r)}|u_n(x)|^{p(x)}dx\leq\frac{c}{\gamma(y)}\rightarrow0,\ \ \text{as}\ \ |y|\rightarrow+\infty.$$

 So, for arbitrary $\epsilon>0$, there exists
$R_{\epsilon}>0$ such that \begin{equation}\label{19}
                             \sup_{|y|\geq R_\epsilon}\int_{B(y,r)}|u_n(x)|^{p(x)}dx\leq\epsilon\ \ \ \text{for all}\ \ n\in\mathbb{N}.
                           \end{equation}
According to Theorem \ref{embedding in frac}, the embedding $$W^{s,p(.),\tilde{p}(.,.)}(B(0,r+R_\epsilon))\hookrightarrow L^{p(x)}(B(0,r+R_\epsilon))$$ is compact and thus $u_n\rightarrow0$ in $L^{p(x)}(B(0,r+R_\epsilon))$,  which implies that
\begin{equation}\label{20}
\int_{B(0,r+R_\epsilon)}|u_n(x)|^{p(x)}dx\rightarrow0,\ \ \ n\rightarrow+\infty.
\end{equation}
For $|y|<R_{\epsilon}$, it comes that $$\int_{B(y,r)}|u_n(x)|^{p(x)}dx\leq \int_{B(0,r+R_\epsilon)}|u_n(x)|^{p(x)}dx.$$
Then, \eqref{20} and the previous inequality yield \begin{equation}\label{zeyid}
      \sup_{|y|<R_\epsilon}\int_{B(y,r)}|u_n(x)|^{p(x)}dx\leq\epsilon\ \ \ \text{for all}\ \ n\in\mathbb{N}.
    \end{equation}
    Combining \eqref{19} and \eqref{zeyid}, we infer that
     $$ \sup_{y\in\mathbb{R}^N}\int_{B(y,r)}|u_n(x)|^{p(x)}dx\leq\epsilon\ \ \ \text{for all}\ \ n\in\mathbb{N}.$$

It follows from Theorem \ref{Lions}, that $u_n\rightarrow0$ in $L^{\alpha(x)}(\mathbb{R}^N)$ for $p\ll\alpha\ll\tilde{p}^{*}_s$. Thus the proof is completed.
\end{proof}

\section{Variational framework and proof of Theorem \ref{thm1}}

In this section, we will apply the results obtained in the previous section to
prove Theorem \ref{thm1}. The problem \eqref{eq1} has a variational structure, namely its
solutions can be found as critical points of a functional, the so-called energy functional.
Hence, our problem can be studied by variational methods.

\begin{definition}\label{weak solution}
 We say that $u\in W^{s,p(.),\tilde{p}(.,.)}(\mathbb{R}^N)$ is a weak solution of \eqref{eq1}, if for every $v\in W^{s,p(.),\tilde{p}(.,.)}(\mathbb{R}^N)$, we have
  \begin{align*}
  &M\big{(}X(u)\big{)}
  \int_{\mathbb{R}^N}\int_{\mathbb{R}^N}\frac{|u(x)-u(y)|^{\tilde{p}(x,y)-2}(u(x)-u(y))(v(x)-v(y))}{|x-y|^{N+s\tilde{p}(x,y)}}dxdy\\&
  =\int_{\mathbb{R}^N}\int_{\mathbb{R}^N}\frac{F(x,u(x))f(y,u(y))v(y)}{|x-y|^{\lambda(x,y)}}dxdy.
  \end{align*}
\end{definition}

Define the functional $I:W^{s,p(.),\tilde{p}(.,.)}(\mathbb{R}^N)\rightarrow\mathbb{R}$ by
$$I(u)=\Phi(u)-\Psi(u),$$ where
$$\Phi(u)=\mathcal{M}\bigg{(}\int_{\mathbb{R}^N}\int_{\mathbb{R}^N}\frac{|u(x)-u(y)|^{\tilde{p}(x,y)}}{\tilde{p}(x,y)|x-y|^{N+s\tilde{p}(x,y)}}dxdy\bigg{)}
=\mathcal{M}\big{(}X(u)\big{)},$$ and
$$\Psi(u)=\frac{1}{2}\int_{\mathbb{R}^N}\int_{\mathbb{R}^N}\frac{F(x,u(x))F(y,u(y))}{|x-y|^{\lambda(x,y)}}dxdy.$$
\begin{lemma}
The functional $I$ is of class $C^1(W^{s,p(.),\tilde{p}(.,.)}(\mathbb{R}^N),\mathbb{R})$ and
\begin{align*}
\langle I^{'}(u),v\rangle&=
M\big{(}X(u)\big{)}
\int_{\mathbb{R}^N}\int_{\mathbb{R}^N}\frac{|u(x)-u(y)|^{\tilde{p}(x,y)-2}(u(x)-u(y))(v(x)-v(y))}{|x-y|^{N+s\tilde{p}(x,y)}}dxdy\\
&-\int_{\mathbb{R}^N}\int_{\mathbb{R}^N}\frac{F(x,u(x))f(y,u(y))v(y)}{|x-y|^{\lambda(x,y)}}dxdy,
\end{align*}
 for all
$u,v\in W^{s,p(.),\tilde{p}(.,.)}(\mathbb{R}^N)$.
\end{lemma}
\begin{proof}
  The proof is completely analogous to [\cite{Alves}, Lemma 3.2] and it is omitted.
\end{proof}
\begin{remark}
 Notice $u\in W^{s,p(.),\tilde{p}(.,.)}(\mathbb{R}^N)$ is a critical point of the functional $I$ if and only if $u$ is a weak solution of \eqref{eq1}.
\end{remark}

We start by setting the following useful lemma.
\begin{lemma}
On $W^{s,p(.),\tilde{p}(.,.)}(\Omega)$,  it holds that
  \begin{equation}\label{important}
    [u]_{s,\tilde{p}(.,.)}\geq \frac{1}{6}|u|_{s,p,\tilde{p}(.,.)}.
  \end{equation}
\end{lemma}
\begin{proof}
Let $u\in W^{s,p(.),\tilde{p}(.,.)}(\Omega)$. If $|u|_{p}\leq [u]_{s,\tilde{p}(.,.)}$, then
\begin{align*}
 \tilde{\rho}\bigg{(}\frac{u}{2 [u]_{s,\tilde{p}(.,.)}}\bigg{)} &= \rho\bigg{(}\frac{u}{2 [u]_{s,\tilde{p}(.,.)}}\bigg{)} + \rho_{\tilde{p}}\bigg{(}\frac{u}{2[u]_{s,\tilde{p}(.,.)}}\bigg{)}\\
  &\leq \rho\bigg{(}\frac{u}{2 |u|_{p}}\bigg{)} + \rho_{\tilde{p}}\bigg{(}\frac{u}{2 [u]_{s,\tilde{p}(.,.)}}\bigg{)}\\
  &\leq \frac{1}{2}+\frac{1}{2}=1.
\end{align*}
So $$[u]_{s,\tilde{p}(.,.)}\geq \frac{1}{2}|u|_{s,p,\tilde{p}(.,.)}.$$
Now, if $|u|_{p}\geq [u]_{s,\tilde{p}(.,.)}$, then $$[u]_{s,\tilde{p}(.,.)}\geq \frac{1}{3}\|u\|_{s,p,\tilde{p}}\geq \frac{1}{6}|u|_{s,p,\tilde{p}}.$$
Therefore, for any $u\in W^{s,p(.),\tilde{p}(.,.)}(\Omega)$,
$$[u]_{s,\tilde{p}(.,.)}\geq \frac{1}{6}|u|_{s,p,\tilde{p}(.,.)}.$$ Thus the proof.
\end{proof}

To overcome the lack of compactness of our problem, we restrict ourselves to $W^{s,p(.),\tilde{p}(.,.)}_{rad}$.

\begin{lemma}\label{geo}
  Suppose that the assumptions in Theorem \ref{thm1} hold. Then, the functional $I$ satisfies the following properties:
  \begin{enumerate}
    \item [(i)] There exists $\delta,\beta>0$, such that $I(u)\geq\beta>0$ for all $u\in W^{s,p(.),\tilde{p}(.,.)}_{rad}(\mathbb{R}^N)$, $\|u\|_{s,p,\tilde{p},\mathbb{R}^N}=\delta$.
    \item [(ii)] There exists $e\in  W^{s,p(.),\tilde{p}(.,.)}_{rad}(\mathbb{R}^N)$ such that $\|e\|_{s,p,\tilde{p},\mathbb{R}^N}>\delta$ and $I(e)<0$.
  \end{enumerate}
\end{lemma}
\begin{proof}
  $(i)$ Combining $(F_1)$, Proposition \ref{prp3} and Theorem \ref{Strauss}, we get
  \begin{align}\label{1222}
    |F(.,u(.))|_{L^{h_{+}}(\mathbb{R}^N)}&\leq c_2\bigg{(}\int_{\mathbb{R}^N}|u(x)|^{r(x)h_{+}}dx\bigg{)}^{\frac{1}{h_+}}\nonumber\\
                                      &\leq c_2\max\big{(}|u|^{r_{+}}_{r(x)h_{+}},|u|^{r_{-}}_{r(x)h_{+}}\big{)}\nonumber\\
                                      &\leq c_3\max\big{(}|u|_{s,p,\tilde{p},\mathbb{R}^N}^{r_{+}},|u|_{s,p,\tilde{p},\mathbb{R}^N}^{r_{-}}\big{)}
  \end{align}
  where the constant $c_2,c_3>0$ are independent of $u$. Similarly for $u\in W^{s,p(.),\tilde{p}(.,.)}_{rad}(\mathbb{R}^N)$ we can check that $F(x,u(x))\in L^{h_{-}}(\mathbb{R}^N)$. Hence from Theorem \ref{Hardy} and \eqref{1222}, we infer that
  \begin{align}\label{100}
    \bigg{|}\int_{\mathbb{R}^N}\int_{\mathbb{R}^N}\frac{F(x,u(x))F(y,u(y))}{|x-y|^{\lambda(x,y)}}dxdy\bigg{|} & \leq C\big{(} |F(.,u(.))|_{h_{+},\mathbb{R}^N}^2+|F(.,u(.))|_{h_{-},\mathbb{R}^N}^2\big{)}\nonumber\\
             &\leq c_4 \max\big{(}|u|_{s,p,\tilde{p},\mathbb{R}^N}^{2r_{+}},|u|_{s,p,\tilde{p},\mathbb{R}^N}^{2r_{-}}\big{)}
  \end{align}
  for some constants $C,c_4>0$ independent of $u$. By Proposition \ref{prp5} and \eqref{important},  for $|u|_{s,p,\tilde{p},\mathbb{R}^N}<1$, we deduce that
  \begin{align}\label{s7i7}
    \mathcal{M}\bigg{(}\int_{\mathbb{R}^N}\int_{\mathbb{R}^N}\frac{|u(x)-u(y)|^{\tilde{p}(x,y)}}{\tilde{p}(x,y)|x-y|^{N+s\tilde{p}(x,y)}}dxdy\bigg{)} & \geq m_0 [u]_{s,\tilde{p}(.,.)}^{\tilde{p}_+}\geq\frac{m_0}{6^{\tilde{p}_+}} |u|_{s,p,\tilde{p},\mathbb{R}^N}^{\tilde{p}_+}.
  \end{align}
  Using \eqref{100} and \eqref{s7i7}, for $|u|_{s,p,\tilde{p},\mathbb{R}^N}<1$, we obtain
  \begin{align*}
     I(u)& =\mathcal{M}\bigg{(}\int_{\mathbb{R}^N}\int_{\mathbb{R}^N}\frac{|u(x)-u(y)|^{\tilde{p}(x,y)}}{\tilde{p}(x,y)|x-y|^{N+s\tilde{p}(x,y)}}dxdy\bigg{)}
-\frac{1}{2}\int_{\mathbb{R}^N}\int_{\mathbb{R}^N}\frac{F(x,u(x))F(y,u(y))}{|x-y|^{\lambda(x,y)}}dxdy\\
&\geq \frac{m_0}{\tilde{p}_{+}6^{\tilde{p}_+}}|u|_{s,p,\tilde{p},\mathbb{R}^N}^{\tilde{p}_{+}}-c_5\max\big{(}|u|_{s,p,\tilde{p},\mathbb{R}^N}^{2r_{+}},|u|_{s,p,\tilde{p},\mathbb{R}^N}^{2r_{-}}\big{)}\\
&\geq \frac{m_0}{\tilde{p}_{+}6^{\tilde{p}_+}} |u|_{s,p,\tilde{p},\mathbb{R}^N}^{\tilde{p}_{+}}-c_5|u|_{s,p,\tilde{p},\mathbb{R}^N}^{2r_{-}},
  \end{align*}
  where $c_5>0$ independent of $u$. Since $r_{-}>\tilde{p}_{+}$, we can choose $\delta>0$ sufficiently small such that $I(u)\geq\beta>0$ for all $u\in W^{s,p(.),\tilde{p}(.,.)}_{rad}(\mathbb{R}^N)$ with $|u|_{s,p,\tilde{p},\mathbb{R}^N}=\delta$.\\

  $(ii)$ The condition $(F_2)$ implies that $$F(x,t)\geq l_1 t^{\theta/2}\ \ \text{for all}\ \ (x,t)\in\mathbb{R}^N\times\mathbb{R}\ \ \text{and}\ \ t\geq l_2,$$
  where $l_1,l_2>0$. Considering a nonnegative function $\varphi\in C_0^{\infty}(\mathbb{R}^N)\setminus\{0\}$, using Proposition \ref{prp5} and the last inequality, we deduce that \begin{align*}
        I(t\varphi)& =\mathcal{M}\bigg{(}\int_{\mathbb{R}^N}\int_{\mathbb{R}^N}\frac{|t\varphi(x)-t\varphi(y)|^{\tilde{p}(x,y)}}{\tilde{p}(x,y)|x-y|^{N+s\tilde{p}(x,y)}}dxdy\bigg{)}
-\frac{1}{2}\int_{\mathbb{R}^N}\int_{\mathbb{R}^N}\frac{F(x,t\varphi(x))F(y,t\varphi(y))}{|x-y|^{\lambda(x,y)}}dxdy\\
&\leq \frac{\mathcal{M}(1)}{\tilde{p}_{+}}\bigg{(}\int_{\mathbb{R}^N}\int_{\mathbb{R}^N}\frac{|t\varphi(x)-t\varphi(y)|^{\tilde{p}(x,y)}}{|x-y|^{N+s\tilde{p}(x,y)}}dxdy\bigg{)}^{\alpha}-\frac{l_1^2 t^{\theta}}{2}\int_{\mathbb{R}^N}\int_{\mathbb{R}^N}\frac{|\varphi(x)|^{\theta/2}|\varphi(y)|^{\theta/2}}{|x-y|^{\lambda(x,y)}}dxdy\\
&\leq \frac{\mathcal{M}(1)}{\tilde{p}_{+}}t^{\alpha\tilde{p}_{+}}|\varphi|_{s,p,\tilde{p},\mathbb{R}^N}^{\alpha\tilde{p}_{+}}-\frac{l_1^2 t^{\theta}}{2}\int_{\mathbb{R}^N}\int_{\mathbb{R}^N}\frac{|\varphi(x)|^{\theta/2}|\varphi(y)|^{\theta/2}}{|x-y|^{\lambda(x,y)}}dxdy.
       \end{align*}
       Since $\alpha\tilde{p}_{+}<\theta$, it follows that $I(t\varphi)\rightarrow-\infty$ as $t\rightarrow+\infty$. This finishes the proof.
\end{proof}

\begin{lemma}\label{radialsolution}
  Assume that $(M_1)-(M_2)$ and $(F_1)-(F_2)$ hold. Then the functional $I|_{W^{s,p(.),\tilde{p}(.,.)}_{rad}}$  admits a non-trivial critical point $u_1$ in
$W^{s,p(.),\tilde{p}(.,.)}_{rad}$.
\end{lemma}
\begin{proof}
   Note that Lemma $4.4$ in \cite{Biswas} guarantees that $I$ satisfies the Palais-Smale condition, together with Lemma \ref{geo} and by applying the Mountain pass Theorem, there exists $u_1\in W^{s,p(.),\tilde{p}(.,.)}_{rad}(\mathbb{R}^N)$, a critical point of $I|_{W^{s,p(.)}_{rad}}$, that is
$$\langle I^{'}(u_1),\varphi\rangle=0\ \ \ \text{for any}\ \ \ \varphi\in W^{s,p(.),\tilde{p}(.,.)}_{rad}(\mathbb{R}^{N}).$$
\end{proof}

The function $u_1$ given by Lemma \ref{radialsolution} is a critical point of $I|_{W^{s,p(.),\tilde{p}(.,.)}_{rad}}$. We have to show that $u_1$ is a  critical point of $I$, that is in sense of definition \ref{weak solution}. To this aim we will use the well known principle of symmetric criticality, (\cite{Palais}).\\

Let $(E,\|.\|_E)$ be a reflexive Banach space. Suppose that $G$ is a subgroup of isometries $g:E\rightarrow E$.
Consider the $G$-invariant closed subspace of $E$ $$\Sigma=\{u\in E:\ gu=u\ \text{for all}\ g\in G\}.$$

\begin{lemma}(Proposition 3.1 of \cite{daniel} or Lemma 5.4 of \cite{Pucci1})\label{psc}
  Let $E$,$G$ and $\Sigma$ be as before and let $J$ be a $C^1$ functional defined on $E$ such that $J\circ g=J$ for all $g\in G$. Then $u\in \Sigma$ is a critical point of $J$ if and only if $u$ is a critical point of $J|_{\Sigma}$.
\end{lemma}

Finally, we have all the ingredients to complete the proof of Theorem \ref{thm1}.

\begin{proof}[{\bf Proof of Theorem \ref{thm1}}]

 Let $SO(N)$ denote the special orthogonal group, that is
$$SO(N)=\{A\in \mathcal{M}_{N\times N}(\mathbb{R}):\ \ A^{t}A=I_N\ \text{and}\ det(A)=1\}.$$
Consider the following subgroup of linear operators of $W^{s,p(.),\tilde{p}(.,.)}(\mathbb{R}^{N})$ in itself
$$G=\bigg{\{}a:W^{s,p(.),\tilde{p}(.,.)}(\mathbb{R}^{N})\rightarrow W^{s,p(.),\tilde{p}(.,.)}(\mathbb{R}^{N}):\ \ au=u\circ A,\ \text{where}\ A\in SO(N)\bigg{\}}.$$
We have $$W^{s,p(.),\tilde{p}(.,.)}_{rad}(\mathbb{R}^{N})=\{u\in W^{s,p(.),\tilde{p}(.,.)}(\mathbb{R}^{N}):\ gu=u\ \text{for all}\ g\in G\}.$$

Let's prove that $G$ is a subgroup of isometries of $W^{s,p(.),\tilde{p}(.,.)}(\mathbb{R}^{N})$: Fixe $u$ in $W^{s,p(.),\tilde{p}(.,.)}(\mathbb{R}^{N})$, since $|x-y|=|A(x-y)|=|Ax-Ay|=|x^{'}-y^{'}|$, $det\ A=1$ and $p,\tilde{p}$ are radial functions, for all $a\in G$, then
\begin{align*}
  |au|_{s,p,\tilde{p},\mathbb{R}^N}&=\inf\bigg{\{}\lambda>0:\ \int_{\mathbb{R}^N}\bigg{|}\frac{u(Ax)}{\lambda}\bigg{|}^{p(x)}dx+
  \int_{\mathbb{R}^N}\int_{\mathbb{R}^N}\frac{|u(Ax)-u(Ay)|^{\tilde{p}(x,y)}}{\lambda^{\tilde{p}(x,y)}|x-y|^{N+s\tilde{p}(x,y)}}dxdy<1\bigg{\}}\\
  &=\inf\bigg{\{}\lambda>0:\ \int_{\mathbb{R}^N}\bigg{|}\frac{u(x^{'})}{\lambda}\bigg{|}^{p(A^{t}x^{'})}dx^{'}+
  \int_{\mathbb{R}^N}\int_{\mathbb{R}^N}\frac{|u(x^{'})-u(y^{'})|^{\tilde{p}(A^{t}x^{'},A^{t}y^{'})}}{\lambda^{\tilde{p}(A^{t}x^{'},A^{t}y^{'})}
  |x^{'}-y^{'}|^{N+s\tilde{p}(A^{t}x^{'},A^{t}y^{'})}}dx^{'}dy^{'}<1\bigg{\}}\\
  &=\inf\bigg{\{}\lambda>0:\ \int_{\mathbb{R}^N}\bigg{|}\frac{u(x^{'})}{\lambda}\bigg{|}^{p(x^{'})}dx^{'}+
  \int_{\mathbb{R}^N}\int_{\mathbb{R}^N}\frac{|u(x^{'})-u(y^{'})|^{\tilde{p}(x^{'},y^{'})}}{\lambda^{\tilde{p}(x^{'},y^{'})}
  |x^{'}-y^{'}|^{N+s\tilde{p}(x^{'},y^{'})}}dx^{'}dy^{'}<1\bigg{\}}\\
  &=|u|_{s,p,\tilde{p},\mathbb{R}^N}.
\end{align*}

To apply Lemma \ref{psc} to the functional $I$ , we need to show that $I \circ a = I$ for all $a \in G$.
Fixed $u\in W^{s,p(.),\tilde{p}(.,.)}(\mathbb{R}^{N})$, since $p,\tilde{p}$ and $f$ are radial functions, for all $a \in G$ it comes that
\begin{align*}
  (I\circ a)(u) & =\mathcal{M}\bigg{(}\int_{\mathbb{R}^N}\int_{\mathbb{R}^N}\frac{|u(Ax)-u(Ay)|^{\tilde{p}(x,y)}}{\tilde{p}(x,y)|x-y|^{N+s\tilde{p}(x,y)}}dxdy\bigg{)}
-\frac{1}{2}\int_{\mathbb{R}^N}\int_{\mathbb{R}^N}\frac{F(x,u(Ax))F(y,u(Ay))}{|x-y|^{\lambda(x,y)}}dxdy\\
&=\mathcal{M}\bigg{(}\int_{\mathbb{R}^N}\int_{\mathbb{R}^N}\frac{|u(x^{'})-u(y^{'})|^{\tilde{p}
(A^{t}x^{'},A^{t}y^{'})}}{\tilde{p}(A^{t}x^{'},A^{t}y^{'})|x^{'}-y^{'}|^{N+s\tilde{p}(A^{t}x^{'},A^{t}y^{'})}}dxdy\bigg{)}
-\frac{1}{2}\int_{\mathbb{R}^N}\int_{\mathbb{R}^N}\frac{F(x,u(x^{'}))F(y,u(y^{'}))}{|x^{'}-y^{'}|^{\lambda(A^{t}x,A^{t}y)}}dxdy\\
&=I(u).
\end{align*}

Then, Lemma \ref{psc} implies that $u_1$ is a critical point of $I$ in the whole space $W^{s,p(.),\tilde{p}(.,.)}(\mathbb{R}^{N})$. Thus, $u_1$ is a solution
of \eqref{eq1} in the sense of definition \ref{weak solution}.
\end{proof}

\end{document}